\documentclass[oneside]{amsart}
\usepackage{amsmath,amssymb,color,amsthm,graphicx,cite}
\usepackage[T1]{fontenc}
\usepackage[utf8]{inputenc}
\usepackage{color,bm}
\usepackage{enumerate} 
\usepackage[backref]{hyperref}
\RequirePackage{luatex85}
\usepackage[all]{xy}
\newtheorem{theorem}{Theorem}[section]
\newtheorem{proposition}[theorem]{Proposition}
\newtheorem{lemma}[theorem]{Lemma}
\newtheorem{corollary}[theorem]{Corollary}
\newtheorem{claim}{Claim}

\newtheorem{question}{Question}

\newtheorem{problem}{Problem}
\newtheorem{conjecture}{Conjecture}
\theoremstyle{definition}
\newtheorem{definition}{Definition}
\newtheorem{main}{Theorem}

\newtheorem{main_cor}[main]{Corollary}

\def\R{\mathbb{R} } 
\def\Z{\mathbb{Z} } 
\def\nbd{neighborhood } 
\def\nbds{neighborhoods } 
\def\R{\mathbb{R} } 
 
\def\Sv{\mathop{\mathrm{Sing}}(v)} 
\def\Pv{\mathop{\mathrm{Per}}(v)} 
\def\Cv{\mathop{\mathrm{Cl}}(v)} 

\def\-{\ominus} 
\def\+{\oplus} 
\def\0{\circ}

\author{Tomoo Yokoyama}
\date{\today}
\address{Department of Mathematics, Faculty of Science, Saitama University, Shimo-Okubo 255, Sakura-ku, Saitama-shi, 338-8570 Japan\\}
\email{tyokoyama@rimath.saitama-u.ac.jp}
\thanks{The author was partially supported by JSPS Grant Number 24K06733}

\makeatletter
\@namedef{subjclassname@2020}{%
  \textup{2020} Mathematics Subject Classification}
\makeatother

\usepackage{color,bm}

\title[Hamiltonian, area-preserving, and non-wandering flows]{Relations among Hamiltonian, area-preserving, and non-wandering flows on compact surfaces}


\begin{document}

\begin{abstract}
This paper gives a topological characterization of Hamiltonian flows with finitely many singular points on compact surfaces, using the concept of ``demi-caract{\'e}ristique'' in the sense of Poincar{\'e}. Furthermore, we clarify the relationships and distinctions among the Hamiltonian, divergence-free, and non-wandering properties for continuous flows, which gives an affirmative answer to the problem posed by Nikolaev and Zhuzhoma under the assumption of finitely many singular points.
\end{abstract}

\maketitle

\section{Introduction}

Area-preserving flows on compact surfaces are one of the basic and classic examples of dynamical systems, also known as locally Hamiltonian flows or equivalently multi-valued Hamiltonian flows. 
The measurable properties of such flows are studied from various aspects  \cite{chaika2021singularity,conze2011cocycles,forni1997solutions,frkaczek2012ergodic,forni2002deviation,kanigowski2016ratner,kulaga2012self,ravotti2017quantitative,ulcigrai2011absence}. 
For instance, the study of area-preserving flows for their connection with solid-state physics and pseudo-periodic topology was initiated by Novikov \cite{novikov1982hamiltonian}. 
The orbits of such flows also arise in pseudo-periodic topology, as hyperplane sections of periodic manifolds (cf. \cite{arnol1991topological,zorich1999leaves}).
Moreover, any Hamiltonian flow on a compact surface is an example of an area-preserving flow. 
The difference between Hamiltonian and area-preserving flows on closed surfaces can be represented by harmonic flows, which are generated by the dual vector fields of harmonic one-forms. 
The topological invariants of Hamiltonian flows with finitely many singular points on compact surfaces are constructed from integrable systems points and dynamical systems points of views, and the structural stability are characterized \cite{bolsinov1999exact,Nikolaenko20,Oshemkov10,sakajo2015transitions,sakajo2018tree,yokoyama2013word}. 
On the other hand, area-preserving flows are examples of non-wandering flows, which is stated in \cite{cobo2010flows,nikolaev1999flows}. 
Non-wandering flows on surfaces are classified and decomposed into elementary cells under the finiteness of singular points, and are topologically characterized, and the topological invariants are constructed \cite{cobo2010flows,nikolaev1998finite,nikolaev2001non,yokoyama2016topological,yokoyama2017decompositions,yokoyama2017genericity}. 

This paper describes equivalence and difference for continuous flows on compact surfaces among Hamiltonian, divergence-free, and non-wandering properties. 

\subsection{Sttements of main results}
To state precisely, recall some concepts. 
By a flow, we mean a continuous $\R$-action. 
A flow is non-wandering if any point is non-wandering. 
Note that every continuous non-wandering flow on a compact surface is topologically equivalent to a $C^\infty$ flow (cf. \cite[Corollary~2.8]{yokoyama2016topological}). 
A flow on a surface is Hamiltonian (resp. area-preserving, divergence-free) if it is topologically equivalent to the flow generated by a smooth Hamiltonian (resp. area-preserving, divergence-free) vector field. 
Area-preserving flows are also called multi-valued Hamiltonian flows. 
Notice that each of the Hamiltonian, area-preserving, divergence-free properties is not preserved by topological equivalence for vector fields but is preserved by one for flows. 
Since any closed $1$-form can be represented locally as a differential of a function, any closed $1$-form on a surface defines an area-preserving flow. 
In \cite{katok1973invariant}(cf. \cite[Theorem~3.5.3]{nikolaev1999flows}), Katok showed that any $C^1$ non-wandering flow on a closed orientable surface whose singular points are saddles is topologically equivalent to the flow generated by a $C^\infty$ area-preserving vector field. 
In \cite{nikolaev1999flows}, Nikolaev and Zhuzhoma 
posed the following open problem (\cite[p.62 Problem~5]{nikolaev1999flows}): 

\begin{problem}\label{prob:01}
To prove that any $C^1$ non-wandering flow on a closed surface is topologically equivalent to the flow generated by a $C^\infty$ ``area-preserving'' vector field. 
\end{problem}

On the other hand, one of the authors also stated ``the non-wandering flows on a surface are given by the closed 1-forms on the surface'' in the introduction and an essentially same statement in the abstract of the paper \cite{nikolaev2001non} published shortly after \cite{nikolaev1999flows} without proof or citation. 
Since a closed 1-forms on a surface defines an area-preserving flow, which is equivalent to a divergence-free flow by Liouville’s theorem, the statement is equivalent to the following conjecture: 

\begin{conjecture}\label{conj:01}
Every flow on a compact surface is non-wandering if and only if it is divergence-free. 
\end{conjecture}

%
%

We answer positively Conjecture~\ref{conj:01} and Problem~\ref{prob:01} under the finiteness of singular points. 
In other words, we have the following equivalence among non-wandering, area-preserving, and divergence-free flows, which is a generalization of the above Katok's result. 

\begin{main}\label{nw_div_free}
The following are equivalent for a flow with finitely many singular points on a compact surface: 
\\
{\rm(1)} The flow is non-wandering. 
\\
{\rm(2)} The flow is divergence-free. 
\\
If $S$ is orientable, then each of the above equivalent conditions is also equivalent to the following statement:
\\
{\rm(3)} The flow is area-preserving. 
\end{main}

We also consider the following closely related question, posed by M. Asaoka~\cite{asaoka2016}, which serves as a motivating example for Problem~\ref{prob:01}.

%
\begin{question}
What is the difference between Hamiltonian flows and non-wandering flows on surfaces?
\end{question}

To answer this question, since the difference between Hamiltonian and multi-valued Hamiltonian flows is the same as that between exact and closed $1$-forms, we show that the obstruction is exactly the existence of directed cycles on a certain quotient space and non-trivial recurrent orbits as follows.

\begin{main}\label{nw_ham}
The following are equivalent for a flow with finitely many singular points on an orientable compact surface $S$: 
\\
{\rm(1)} The flow $v$ is Hamiltonian. 
\\
{\rm(2)} The flow $v$ is non-wandering, there are no locally dense orbits, and the extended orbit space $S/v_{\mathrm{ex}}$ (see the definition in \S ~\ref{def:extended_orbit}) is a directed graph without directed cycles. 
\end{main}

Note that an extended orbit for the flow in the previous theorem is equivalent to the concept of  ``demi-caract{\'e}ristique'' in the sense of Poincar{\'e} \cite{poincare1881memoire} and so that the extended orbit space is the quotient space collapsing ``demi-caract{\'e}ristiques'' into singletons.  
Moreover, this quotient space is an analogous concept to a Reeb graph of a function.
The non-existence of directed cycles can not be replaced by the non-existence of closed transversals in the previous theorem (see the third example as in Figure~\ref{non_dc} in \S \ref{example}).
There is no difference in the sphere between Hamiltonian and non-wandering flows under the finiteness of singular points as follows. 

\begin{main_cor}\label{div_free_sphere}
Let $v$ be a flow with finitely many singular points on a compact surface $S$ contained in a sphere. 
The following statements are equivalent: 
\\
$(1)$ The flow $v$ is non-wandering. 
\\
$(2)$ The flow $v$ is area-preserving. 
\\
$(3)$ The flow $v$ is divergence-free.
\\
$(4)$ The flow $v$ is Hamiltonian.
\end{main_cor}

The present paper consists of five sections.
In the next section, we define fundamental concepts as preliminaries. 
The equivalence among divergence-free flows, area-preserving flows, and multi-valued Hamiltonians is stated, which is essentially Liouville's theorem (Lemma~\ref{lem:eq01}).
In \S~3, a characterization of Hamiltonian flows with finitely many singular points on a compact surface, as well as the equivalence among non-wandering flows, area-preserving flows, and divergence-free flows on compact surfaces under the finiteness condition of singular points, relies on the key lemma, which is demonstrated in the next section.
In particular, Theorem~\ref{nw_div_free}, Theorem~\ref{nw_ham}, and Corollary~\ref{div_free_sphere} are proved.
The key lemma used in \S~3 is established in \S~4.
In the final section, an example is provided to illustrate the necessity of the non-existence of directed cycles, as well as a case of a non-wandering flow that may not be Hamiltonian.

\section{Preliminaries}

In this section, we define fundamental concepts from topology and dynamical systems. 

\subsection{Notion of topology}

Denote by $\overline{A}$ the closure of a subset $A$ of a topological space, by $\mathop{\mathrm{int}} A$ the interior of $A$, and by $\partial A := \overline{A} - \mathop{\mathrm{int}} A$ the boundary of $A$, where $B - C$ is used instead of the set difference $B \setminus C$ when $B \subseteq C$.
We define the border $\partial^-  A$ by $A - \mathrm{int}A$ of $A$. 
A {\bf boundary component} of a subset $A$ is a connected component of the boundary of $A$. 
A subset is {\bf locally dense} if its closure has a nonempty interior. 
%
By a {\bf surface},  we mean a two-dimensional paracompact manifold, that does not need to be orientable.

\subsubsection{Directed topological graphs}
A {\bf directed topological graph} is a topological realization of a 1-dimensional simplicial complex with a directed structure on edges. 
A {\bf directed cycle} in a directed topological graph is an embedded cycle whose edges are oriented in the same direction. 

\subsubsection{Curves and arc}
A {\bf curve} is a continuous mapping $C: I \to X$ where $I$ is a non-degenerate connected subset of a circle $\mathbb{S}^1$.
A curve is simple if it is injective.
We also denote by $C$ the image of a curve $C$.
Denote by $\partial C := C(\partial I)$ the boundary of a curve $C$ if $C$ can be extend into a continuous map whose domain is $I \cup \partial I$, where $\partial I$ is the boundary of $I \subset \mathbb{S}^1$. 
Put $\mathop{\mathrm{int}} C := C \setminus \partial C$ if $\partial C$ is defined. 
A simple curve is a simple closed curve if its domain is $\mathbb{S}^1$ (i.e. $I = \mathbb{S}^1$).
An {\bf arc} is a simple curve whose domain is a non-degenerate interval. 


\subsubsection{Reeb graph of a function on a topological space}
For a function $f \colon  X \to \R$ on a topological space $X$, the {\bf Reeb graph} of a function $f \colon  X \to \R$ on a topological space $X$ is a quotient space $X/\mathop{\sim}_{\mathrm{Reeb}}$ defined by $x \sim_{\mathrm{Reeb}} y$ if there are a number $c \in \R$ and a connected component of $f^{-1}(c)$ which contains $x$ and $y$.

\subsubsection{Directed topological graphs}
A {\bf directed topological graph} is a topological realization of a 1-dimensional simplicial complex with a directed structure on edges. 
A {\bf directed cycle} in a directed topological graph is an embedded cycle whose edges are oriented in the same direction. 
A directed topological graph is a {\bf directed surface graph} if it is realized in a surface.

\subsection{Notion of dynamical systems}
We recall some basic concepts. A good reference for most of what we describe is a book by S. Aranson, G. Belitsky, and E. Zhuzhoma \cite{aranson1996introduction}. 

By a {\bf flow}, we mean a continuous $\mathbb{R}$-action on a surface. 
Let $v \colon  \R \times S \to S$ be a flow on a surface $S$. 
For $t \in \R$, define $v^t : S \to S$ by $v^t := v(t, \cdot )$.
For a point $x$ of $S$, we denote by $O(x)$ the {\bf orbit} of $x$ (i.e. $O(x) := \{ v^t(x) \mid t  \in \R \}$). 
An orbit arc is a non-degenerate arc contained in an orbit. 
A subset of $S$ is said to be {\bf invariant} (or saturated) if it is a union of orbits. 
The {\bf saturation} $v(A) = \bigcup_{x \in A} O(x)$ of a subset $A \subseteq S$ is the union of orbits intersecting $A$. 

A point $x$ of $S$ is {\bf singular} if $x = v^t(x)$ for any $t \in \R$ and is {\bf periodic} if there is a positive number $T > 0$ such that $x = v^T(x)$ and  $x \neq v^t(x)$ for any $t \in (0, T)$. 
A point is {\bf closed} if it is either singular or periodic. 
Denote by $\mathop{\mathrm{Sing}}(v)$ (resp. $\mathop{\mathrm{Per}}(v)$, $\mathop{\mathrm{Cl}}(v)$) the set of singular (resp. periodic, closed) points. 
A point is wandering if there are its neighborhood $U$ and a positive number $N$ such that $v^t(U) \cap U = \emptyset$ for any $t > N$. Then, such a neighborhood is called a {\bf wandering domain}. 
A point is {\bf non-wandering} if it is not wandering (i.e. for any neighborhood $U$ and for any positive number $N$, there is a number $t \in \mathbb{R}$ with $|t| > N$ such that $v^t(U) \cap U \neq \emptyset$).
A flow is {\bf non-wandering} if each point is non-wandering. 

For a point $x \in S$, define the $\omega$-limit set $\omega(x)$ and the $\alpha$-limit set $\alpha(x)$ of $x$ as follows: $\omega(x) := \bigcap_{n\in \mathbb{R}}\overline{\{v^t(x) \mid t > n\}} $, $\alpha(x) := \bigcap_{n\in \mathbb{R}}\overline{\{v^t(x) \mid t < n\}} $. 
For an orbit $O$, define $\omega(O) := \omega(x)$ and $\alpha(O) := \alpha(x)$ for some point $x \in O$.
Note that an $\omega$-limit (resp. $\alpha$-limit) set of an orbit is independent of the choice of a point in the orbit. 
A {\bf separatrix} is a non-singular orbit whose $\alpha$-limit or $\omega$-limit set is a singular point.
A separatrix is {\bf semi-connecting} if one of its $\omega$-limit set and $\alpha$-limit sets is a singular point.
A semi-connecting separatrix is {\bf connecting} if each of its $\omega$-limit set and $\alpha$-limit sets is a singular point.

A point $x$ of $X$ is {\bf recurrent} (cf. \cite{marzougui2002area}) if $x \in \omega(x) \cup \alpha(x)$.
Notice that some authors refer to the property $x \in \omega(x) \cap \alpha(x)$ as recurrence (cf. \cite{nikolaev2013foliations}). 
An orbit is singular (resp. periodic, closed, non-wandering, recurrent) if it consists of singular (resp. periodic, closed, non-wandering, recurrent) points. 
A non-closed recurrent orbit is also called a non-trivial recurrent orbit. 
A quasi-minimal set is an orbit closure of a non-closed recurrent orbit. 
The closure of a non-closed recurrent orbit is called a {\bf Q-set}. 


\subsubsection{Topological properties of orbits}

An orbit $O$ is {\bf proper} if there is its \nbd $U$ with $\overline{O} \cap U = O$. 
Note that an orbit $O$ is proper if and only if it is an embedded submanifold. 
Moreover, any closed orbit is proper. 
An orbit is {\bf exceptional} if it is neither proper nor locally dense. 
A point is proper (resp. locally dense, exceptional) if its orbit is proper (resp. locally dense, exceptional).
Denote by $\mathrm{LD}(v)$ (resp. $\mathrm{E}(v)$, $\mathrm{P}(v)$) the union of locally dense orbits (resp. exceptional orbits, non-closed proper orbits). 
%
We have the following observation. 
\begin{lemma}[cf. Lemma~2.1\cite{yokoyama2023topological}]\label{lem:decomp}
The following statements hold for a flow $v$ on a paracompact manifold $S$: 
\\
{\rm(1)} The union $\mathrm{P}(v)$ of non-closed proper orbits is the set of non-recurrent points. 
\\
{\rm(2)} The subset $\mathop{\mathrm{Cl}}(v) \sqcup \mathrm{LD}(v) \sqcup \mathrm{E}(v)$ is the set of recurrent points. 
\\
{\rm(3)} The subset $\mathrm{R}(v) = \mathrm{LD}(v) \sqcup \mathrm{E}(v)$ is the union of non-proper orbits. 
\\
{\rm(4)} $S = \mathop{\mathrm{Sing}}(v) \sqcup \mathop{\mathrm{Per}}(v) \sqcup \mathrm{P}(v) \sqcup \mathrm{R}(v) = \mathop{\mathrm{Cl}}(v) \sqcup \mathrm{P}(v) \sqcup \mathrm{R}(v)$. 
\end{lemma}


\subsubsection{Topological equivalence}
A flow $v$ on a surface $S$ is {\bf topologically equivalent} to a flow $w$ on a surface $T$ if there is a homeomorphism $h \colon S \to T$ such that the image of any orbit of $v$ is an orbit of $w$ with preservation of the direction in time. 

To classify and analyze non-invariant sets (e.g. small neighborhoods of saddles) using techniques from foliation theory, we define the local equivalence of flows as follows.

\begin{definition}
The restriction of a flow $v$ to a subset $U$ of $S$ is ({\bf locally}) {\bf topologically equivalent} to one of a flow $w$ to a subset $V$ on a topological space $T$ if there is a homeomorphism $h \colon U \to V$ satisfying the following conditions: 
\\
{\rm(1)} The images via $h$ of any orbit arcs of $v$ in $U$ are orbit arcs of $w$. 
\\
{\rm(2)} The inverse images via $h$ of any orbit arcs of $w$ in $V$ are orbit arcs of $v$. 
\\
{\rm(3)} Both $h$ and $h^{-1}$ preserve the directions of orbit arc of $v$ and $w$. 
\end{definition}

In the previous definition, the restriction $v\vert_U$ is also said to be topologically equivalent to $w\vert_V$ via $h$. 
Notice that, though the topological equivalence is a relation between flows, the locally topological equivalence is a relation between the restrictions to subsets.

\subsubsection{Types of singular points}
An isolated singular point of $S$ is a {\bf center} if there is an invariant open disk $U$ containing it such that the restriction to $U$ is locally topologically equivalent to the flow on $\{ (x,y) \in \R^2 \mid x^2 + y^2 < 1 \}$ generated by a vector field $(-y,x)$ as in the left on Figure~\ref{multi-saddles}. 
\begin{figure}
\begin{center}
\includegraphics[scale=0.325]{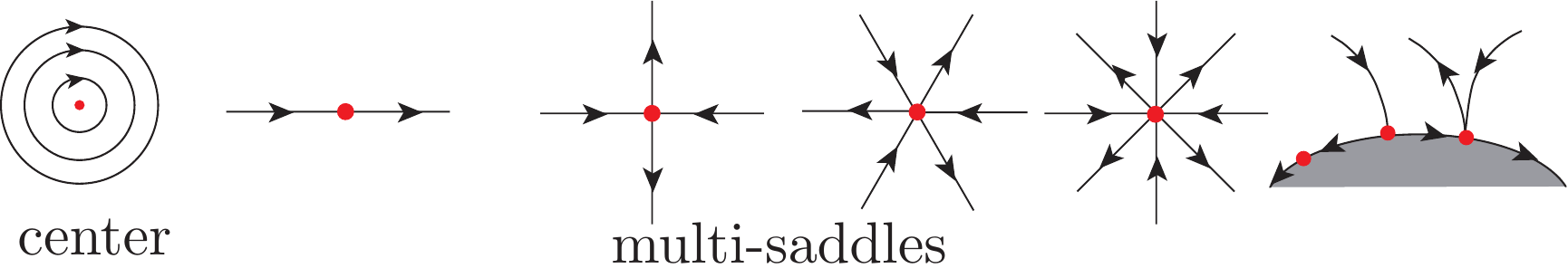}
\end{center}
\caption{A center and examples of multi-saddles}
\label{multi-saddles}
\end{figure} 

A $\partial$-$k$-saddle (resp. $k$-saddle) is an isolated singular point on (resp. outside of) $\partial S$ with exactly $(2k + 2)$-separatrices, counted with multiplicity, as in Figure~\ref{multi-saddles}.

\begin{definition}
A {\bf multi-saddle} is a $k$-saddle or a $\partial$-$(k/2)$-saddle for some $k \in \mathbb{Z}_{\geq 0}$.
\end{definition}
A $1$-saddle is topologically an ordinary saddle, and a $\partial$-$(1/2)$-saddle is topologically a $\partial$-saddle.
%

Notice that that every non-wandering flow with finitely many singular points on a compact surface consists of multi-saddles and centers, because of \cite[Theorem 3]{cobo2010flows}.

\subsubsection{Multi-saddle connection diagrams}
The {\bf multi-saddle connection diagram} $D(v)$ is the union of multi-saddles and separatrices between multi-saddles.
A {\bf multi-saddle connection} is a connected component of the multi-saddle connection diagram.
Note that a multi-saddle connection is also called a poly-cycle.

\subsubsection{Flow boxes, periodic annuli, and transverse annuli}
We define the following local structures. 

\begin{definition}
A subset $U$ is a {\bf trivial flow box} if there are bounded non-degenerate intervals $I, J$ and a homeomorphism $h \colon I \times J \to U$ such that the image $h(I \times \{ t \})$ for any $t \in J$ is an orbit arc as in the left on Figure~\ref{flow-boxes}. 
\end{definition}

\begin{definition}
An annulus $A$ is a {\bf periodic annulus} if there is an interval $I \subseteq [1/2,1]$ such that the restriction of the flow on $A$ is locally topologically equivalent to the flow on $\{ (x,y) \in \mathbb{R}^2 \mid x^2 + y^2 \in I \}$ generated by a vector field $(-y,x)$, as in the middle of  Figure~\ref{flow-boxes}. 
\end{definition}

\begin{definition}
An annulus $A$ is a {\bf transverse annulus} if there is an interval $I \subseteq [1/2,1]$ such that the restriction of the flow on $A$ is locally topologically equivalent to the flow on $\{ (x,y) \in \mathbb{R}^2 \mid x^2 + y^2 \in I \}$ generated by a vector field $(x,y)$, as in the right on Figure~\ref{flow-boxes}. 
\end{definition}

\begin{figure}
\begin{center}
\includegraphics[scale=0.4]{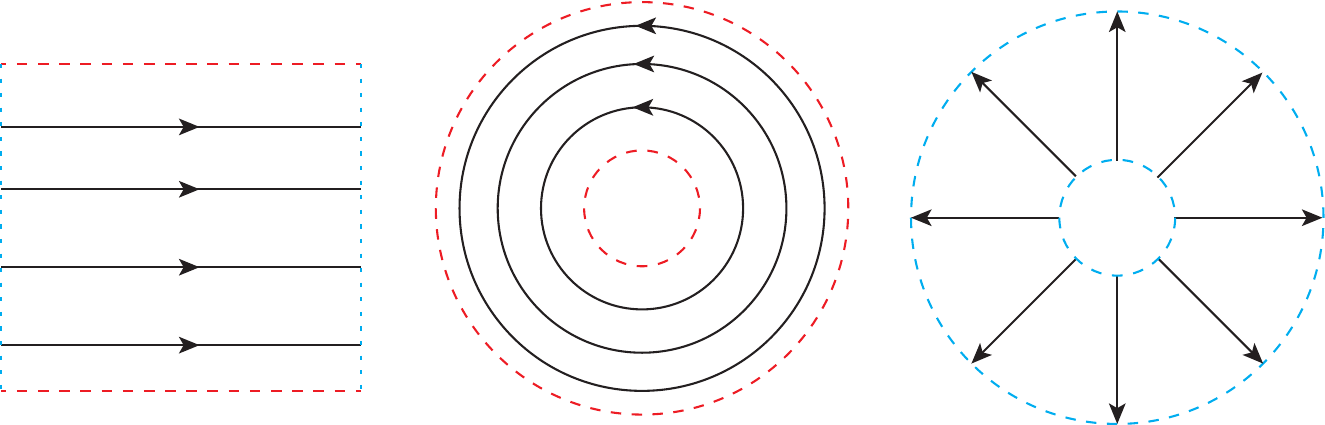}
\end{center}
\caption{An open trivial flow box, an open periodic annulus, and an open transverse annulus}
\label{flow-boxes}
\end{figure}

\subsubsection{Periodic invariant subsets and rotating surfaces}
A {\bf periodic torus} (resp. {\bf periodic Klein bottle}, {\bf periodic M{\"o}bius band}) is a torus (resp. Klein bottle, M{\"o}bius band) consisting of periodic orbits. 

A {\bf rotating sphere} (resp. {\bf rotating projective plane}) is topologically equivalent to an isometric flow on a sphere (resp. projective plane) that consists of two centers (resp. one center) and periodic orbits. 
An open (resp. closed) {\bf center disk} is topologically equivalent to an isometric flow on an open (resp. closed) disk that consists of a center and periodic orbits. 
In other words, roughly speaking, by pasting one center (resp. two centers) with boundary components of an open periodic annulus, we obtain an open center disk (resp. rotating sphere). 
Pasting one center with the boundary of an open periodic M{\"o}bius band, we obtain a rotating projective plane.

\subsection{Characterizations of non-wandering flows}

Recall the following characterization of non-wandering flow. 

\begin{lemma}[Theorem 2.5 in \cite{yokoyama2016topological} and Proposition 8.1 in  \cite{yokoyama2021generalization}]\label{lem3-11}
Let $v$ be a flow on a compact surface $S$. 
The following statements are equivalent: 
\\
$(1)$ The flow $v$ is non-wandering. 
\\
$(2)$ $\mathop{\mathrm{int}}\mathrm{P}(v) = \emptyset$. 
\\
$(3)$ $\mathop{\mathrm{int}}\mathrm{P}(v)) \sqcup \mathrm{E}(v) = \emptyset$ $(\mathrm{i.e.}$ $S = \mathop{\mathrm{Cl}}(v) \sqcup \partial^- \mathrm{P}(v) \sqcup \mathrm{LD}(v)$. 
\\
$(4)$ $S = \overline{\mathop{\mathrm{Cl}}(v) \sqcup \mathrm{LD}(v)}$. 

In any case, we have 
$\mathrm{P}(v) \sqcup \mathop{\mathrm{Sing}}(v) = \{ x \in S \mid \omega(x) \cup \alpha(x) \subseteq \mathop{\mathrm{Sing}}(v) \}$. 
\end{lemma}

%

\subsection{Fundamental notion related to Hamiltonian flows on a compact surface}

A $C^r$ vector field $X$ for any $r \in \Z_{\geq0}$ on an orientable surface $S$ is {\bf Hamiltonian} if there is a $C^{r+1}$ function $H \colon S \to \mathbb{R}$, called a {\bf Hamiltonian},  such that $dH= \omega(X, \cdot )$ as a one-form, where $\omega$ is a volume form of $S$.
In other words, locally, the Hamiltonian vector field $X$ is defined by $X = (\partial H/ \partial x_2, - \partial H/ \partial x_1)$ for any local coordinate system $(x_1,x_2)$ of a point $p \in S$.
We define the following concept.

\begin{definition}
A flow is {\bf Hamiltonian} if it is topologically equivalent to a flow generated by a smooth Hamiltonian vector field.
\end{definition}

Note that a volume form on an orientable surface is a symplectic form.
It is known that a $C^r$ ($r \geq 1$) Hamiltonian vector field on a compact surface is structurally stable with respect to the set of $C^r$ Hamiltonian vector fields if and only if both each singular point is non-degenerate and each separatrix of a saddle is homoclinic, and each separatrix on a $\partial$-saddle connects a boundary component (see  \cite[Theorem 2.3.8, p. 74]{ma2005geometric}).

The Reeb graph of the Hamiltonian $H$ generating $X$ is a directed topological graph which is also a quotient space of the orbit space of the generating flow. 
In other words, the Reeb graph is obtained from the orbit space by collapsing multi-saddle connections into singletons. 
Note that the orbit space of a Hamiltonian flow is not a directed topological graph. 

\subsubsection{Multi-valued Hamiltonian vector field}
A $C^r$ (resp. $C^\infty$ or smooth) vector field $X$ for any $r \in \Z_{\geq0}$ on a surface $S$ is a {\bf multi-valued Hamiltonian} flow if there are finitely many pairs of charts $(U_i; x_i, y_i)$ with $S = \bigcup U_i$ and $C^{r+1}$ (resp. $C^\infty$) Hamiltonians $H_i$ on the charts such that $X|_{U_i} = (\partial H_i/\partial y_i, - \partial H_i/\partial x_i)$. 
Then the family $\{ H_i \}_i$ is called a {\bf multi-valued Hamiltonian} on $S$. 
A flow is {\bf multi-valued Hamiltonian} if it is topologically equivalent to a flow generated by a smooth multi-valued Hamiltonian vector field.

\subsubsection{Extended orbits and extended orbit spaces}\label{def:extended_orbit}
In the analogy of the collapse, we define the extended orbit to describe ``reduced'' orbit spaces as follows. 
%
By an {\bf extended orbit} of a flow, we mean a multi-saddle connection or an orbit that is not contained in any multi-saddle connection. 
Denote by $O_{\mathrm{ex}}(x)$ the extended orbit containing $x$. 
The quotient space by extended orbits of a flow $v$ on a surface $S$ is called the {\bf extended orbit space} and is denoted by $\bm{{S}/{v_{\mathrm{ex}}}}$. 

In other words, the extended orbit space ${S}/{v_{\mathrm{ex}}}$ is a quotient space $S/\mathop{\sim}$ defined by $x \sim y$ if either $O(x) = O(y) \subset S - D(v)$ or the points $x$ and $y$ are contained in a multi-saddle connection, where $D(v)$ is the multi-saddle connection diagram of $v$.

\subsubsection{Divergence-free flows, area-preserving flows, and multi-valued Hamiltonians}
A $C^r$ vector field $X$ for any $r \in \Z_{>0} \sqcup \{ \infty \}$ on a surface is {\bf area-preserving} if the area form $\mu$ is preserved under pullback (i.e. $(v^{t})^*(\mu) = \mu$) by homeomorphisms $v^t$ for any $t \in R$, where $v$ is the flow generated by $X$. 
A $C^r$ vector field $X$ for any $r \in \Z_{>0} \sqcup \{ \infty \}$ on a surface with a Riemannian metric $g$ is {\bf divergence-free} if $\mathrm{div} X = 0$, where $\mathrm{div}X := \mathop{*} d \mathop{*} g(X, \cdot )$. 
We define the following concepts.

\begin{definition}
A flow is {\bf area-preserving} (resp. {\bf divergence-free}) if it is topologically equivalent to a flow generated by a smooth area-preserving (resp. divergence-free) vector field. 
\end{definition}

Note that the divergence $\mathrm{div} X$ can also be determined by a condition $\mathrm{div}X \mu = L_X \mu$ where $L_X$ is the Lie derivative along $X$. 
Liouville's theorem implies the following observation.

\begin{lemma}\label{lem:eq01}
The following statements are equivalent for a smooth vector field $X$ on a surface $S$: 
\\
{\rm(1)} The vector field $X$ is divergence-free. 
\\
{\rm(2)} The vector field $X$ is multi-valued Hamiltonian. 
\\
If $S$ is orientable, then each of the above equivalent conditions is also equivalent to the following statement {\rm(3)}: 
\\
{\rm(3)} The vector field $X$ is area-preserving. 
\end{lemma}

\begin{proof}
Let $X$ be a smooth vector field on a surface $S$. 
Since any multi-valued Hamiltonian vector field is divergence-free, assertion (2) implies assertion (1).
Conversely, suppose that $X$ is divergence-free. 
By the Cartan's formula, we have $0= \mathrm{div}X \mu = L_X \mu = d( \iota_X \mu)$. 
Therefore the $1$-form $\iota_X \mu$ is a closed $1$-form on $S$. 
By Poincar{\'e} lemma, any closed $1$-form on the plane $\R^2$ is exact and so the closed $1$-form $\iota_X \mu$ coincides with a family of multi-valued Hamiltonians on $S$. 

Suppose that $S$ is orientable. 
Since 
\[
d (v^{t})^*\mu /dt|_{t=t_0} = (v^{t_0})^*(L_X \mu) = (v^{t_0})^*((\mathrm{div} X) \mu),
\] 
we have that $(v^{t})^*\mu = \mu$ for any $t \in \R$ if and only if $\mathrm{div} X = 0$. 
This means that assertions (1) and (3) are equivalent. 
\end{proof}

\section{Characterizations of non-wandering property, divergence free propety, and Hamiltonian property}


In this section, we demonstrate the equivalence between non-wandering flows and divergence-free flows on compact surfaces under the finiteness of singular points, and characterize a Hamiltonian flow with finitely many singular points on a compact surface. 
To achieve these, we state two key lemmas in this section. 
First, we state the following key lemma, which is a generalization of the above Katok's result and is shown below in this section, which is demonstrated in the next section. 

\begin{lemma}\label{lem:nw2div_free}
Every non-wandering flow with finitely many singular points on a compact surface is divergence-free. 
\end{lemma}

The previous lemma implies the following result. 

\begin{proof}[Proof of Theorem~\ref{nw_div_free}]
Lemma~\ref{lem:eq01} implies the equivalence between the divergence-free property and the area-preserving property. 
By Lemma~\ref{lem:nw2div_free}, assertion (1) implies assertion (3). 
Let $v$ be a flow with finitely many singular points on a compact connected surface $S$. 
Suppose that $v$ is divergence-free. 
\cite[Theorem~3]{cobo2010flows} implies that each singular point is either a center or a multi-saddle. 
Since any divergence-free vector field on a compact surface has no wandering domains, the flow $v$ is non-wandering. 
\end{proof}


Moreover, we have another key lemma as follows. 

\begin{lemma}\label{div_free_2_Ham}
Let $v$ be a divergence-free flow with finitely many singular points without locally dense orbits on an orientable compact surface $S$. 
Then the extended orbit space $S/v_{\mathrm{ex}}$ is a directed graph.
Moreover, if the extended orbit space $S/v_{\mathrm{ex}}$ has no directed cycle as a directed graph, then $v$ is Hamiltonian.
\end{lemma}

%
%
%

%
%

By combining the previous lemmas, we conclude that any non-wandering flow with finitely many singular points without locally dense orbits must be Hamiltonian.

\begin{proposition}\label{div_free}
The following statements are equivalent for a flow $v$ with finitely many singular points on a compact connected surface $S$: 
\\
$(1)$ The flow $v$ is a non-wandering flow without locally dense orbits.
\\
$(2)$ The flow $v$ is a divergence-free flow without locally dense orbits. 
\\
Moreover, if $S$ is orientable and the extended orbit space $S/v_{\mathrm{ex}}$ has no directed cycle as a directed graph, then the following assertion is equivalent to one of the above assertions: 
\\
$(3)$ The flow $v$ is Hamiltonian.
\end{proposition}

\begin{proof}
We may assume that $S$ is connected. 
Since Hamiltonian vector fields are divergence-free and the existence of a Hamiltonian implies the non-existence of locally dense orbits, assertion $(3)$ implies assertion $(2)$. 
Suppose that $v$ is a divergence-free flow without locally dense orbits.  
Since any divergence-free vector field on a compact surface has no wandering domains, the flow $v$ is non-wandering. 
This means that assertion $(2)$ implies assertion $(1)$. 
Lemma~\ref{nw_2_div_free} implies that assertion $(1)$ implies assertion $(2)$. 
By Lemma~\ref{div_free_2_Ham}, assertion $(2)$ implies assertion $(3)$ under the non-existence of directed cycles and the orientability of $S$. 
\end{proof}

The previous proposition implies Theorem~\ref{nw_ham}, which is an answer to the question of when a non-wandering flow with finitely many singular points on an orientable compact connected surface becomes Hamiltonian. 
Notice that the previous proof remains true if Lemma~\ref{lem:nw2div_free} in the previous proof is replaced by the weaker statement Lemma~\ref{nw_2_div_free}. 
Moreover, Theorem~\ref{nw_div_free} and Theorem~\ref{nw_ham} demonstrate Corollary~\ref{div_free_sphere} as follows.  

\begin{proof}[Proof of Corollary~\ref{div_free_sphere}]
By Theorem~\ref{nw_div_free} and Theorem~\ref{nw_ham}, it suffices to show that the non-wandering property implies Hamiltonian property.  

Let $v$ be a non-wandering flow with finitely many singular points on a compact surface $S$ contained in a sphere. 
From the non-existence of non-closed recurrent orbits, by Lemma~\ref{nw_finite_type}, the quotient space $\mathbb{S}^2/v_{\mathrm{ex}}$ is a finite directed topological graph. 
Since the quotient space of a sphere is simply connected, the graph $\mathbb{S}^2/v_{\mathrm{ex}}$  has no directed circuit. 
Therefore, the flow $v$ is non-wandering, there are no locally dense orbits, and the extended orbit space $S/v_{\mathrm{ex}}$ is a directed graph without directed cycles. 
Theorem~\ref{nw_ham} implies the assertion. 
\end{proof}

\subsection{Proof of Lemma~\ref{div_free_2_Ham}}

Recall that a periodic orbit is one-sided if it is either a boundary component of a surface or has a small neighborhood which is a M{\"o}bius band.
For a flow $v$ on a surface, let $\mathop{\mathrm{Sing}}_{\mathrm{c}}(v)$ be the set of centers, and $\mathop{\mathrm{Per}}_{1}(v)$ the finite disjoint union of one-sided periodic orbits. 
We have the following decomposition of a non-wandering flow with finitely many singular points on a compact surface, which is essentially a Katok's result \cite{katok1973invariant} (cf. \cite[Theorem~3.1.8]{nikolaev1999flows}) and which is used to prove Lemma~\ref{lem:nw2div_free} and Lemma~\ref{div_free_2_Ham}. 

\begin{lemma}\label{lem:katok_gen}
The following statements hold for any non-wandering flow $v$ with finitely many singular points on a compact surface $S$: 
\\
{\rm(1)} $S - (\mathop{\mathrm{Sing}}(v) \sqcup \mathrm{P}(v))= S - (\mathop{\mathrm{Sing}}_{\mathrm{c}}(v) \sqcup D(v)) = \Pv \sqcup \mathrm{LD}(v)$. 
\\
{\rm(2)} Any connected component of $S -(\mathop{\mathrm{Sing}}_{\mathrm{c}}(v) \sqcup D(v))$ is contained in either $\Pv$ or $\mathrm{LD}(v)$. 
\\
{\rm(3)} Any connected component of $S -(\mathop{\mathrm{Sing}}_{\mathrm{c}}(v) \sqcup D(v))$ consisting $\Pv$ {\rm(resp.} $\mathrm{LD}(v)${\rm)} is either a periodic annulus or a periodic M{\"o}bius band {\rm(resp.} is the intersection of a Q-set and $\mathrm{LD}(v)${\rm)}. 
\\
{\rm(4)} Any connected component of $S -(\mathop{\mathrm{Sing}}_{\mathrm{c}}(v) \sqcup \mathop{\mathrm{Per}}_{1}(v)) \sqcup D(v))$ consisting $\Pv$ is an open periodic annulus. 
\end{lemma}

\begin{proof}
We may assume that $S$ is connected. 
If $S$ is either a periodic torus or a periodic Klein bottle, then $v$ is divergence-free. 
Thus, we may assume that $S$ is neither a periodic torus nor a periodic Klein bottle. 
By \cite[Corollary 2.9]{yokoyama2016topological}, the periodic point set $\Pv$ and the union $\mathrm{LD}(v)$ are open. 
Lemma~\ref{lem3-11} implies that $S = \mathop{\mathrm{Cl}}(v) \sqcup \mathrm{P}(v) \sqcup \mathrm{LD}(v)$ and $\mathop{\mathrm{Sing}}(v) \sqcup \mathrm{P}(v) = \{ x \in S \mid \omega(x) \cup \alpha(x) \subseteq \mathop{\mathrm{Sing}}(v) \}$. 
By the finiteness of singular points, any non-recurrent orbit is a connecting separatrix in $D(v)$. 
This means that the union $\mathrm{P}(v)$ consists of finitely many orbits.
\cite[Theorem~3]{cobo2010flows} implies that each singular point of a non-wandering flow with finitely many singular points on a compact surface is either a center or a multi-saddle. 
Therefore, the disjoint union $\mathop{\mathrm{Sing}}(v) \sqcup \mathrm{P}(v)$ is a closed subset $\mathop{\mathrm{Sing}}_{\mathrm{c}}(v) \sqcup D(v)$ consisting of finitely many orbits. 
The complement $S - (\mathop{\mathrm{Sing}}(v) \sqcup \mathrm{P}(v)) = \Pv \sqcup \mathrm{LD}(v)$ is an open dense subset, and so is a surface with possible boundary. 
Then $S = \overline{S - (\Sv \cup D(v))} = \overline{\Pv \sqcup \mathrm{LD}(v)}$. 

Considering the double of $S$, which is a closed surface, applying \cite[Theorem~3.1.8]{nikolaev1999flows} to the double, any connected component $U$ of $S -(\mathop{\mathrm{Sing}}_{\mathrm{c}}(v) \sqcup D(v)) = \Pv \sqcup \mathrm{LD}(v)$ consisting $\Pv$ (resp. $\mathrm{LD}(v)$) is either a periodic annulus or a periodic M{\"o}bius band (resp. is the intersection of a Q-set and $\mathrm{LD}(v)$). 
Therefore, any connected component of $S -(\mathop{\mathrm{Sing}}_{\mathrm{c}}(v) \sqcup \mathop{\mathrm{Per}}_{1}(v)) \sqcup D(v))$ consisting $\Pv$ is an open periodic annulus. 
\end{proof}

The previous lemma implies the following completeness, which is used to show Lemma~\ref{lem:nw2div_free} in the next section.

\begin{corollary}\label{nw_completeness}
The nonempty multi-saddle connection diagram as a directed surface graph is a complete invariant of any non-wandering flow with finitely many singular points without locally dense orbits on a connected compact surface. 
\end{corollary}

We have the following characterization of a non-wandering flow with finitely many singular points without locally dense orbits. 

\begin{lemma}\label{nw_finite_type}
The following statements are equivalent for a flow $v$ with finitely many singular points on a compact connected surface $S$: 
\\
$(1)$ The flow $v$ is a non-wandering flow without locally dense orbits.
\\
$(2)$ Each connected component of the complement $S - D(v)$ is either a rotating sphere, a center disk, a periodic annulus, a periodic torus, a periodic Klein bottle, a rotating projective plane, or a periodic M{\"o}bius band. 

In any case, $S = \overline{\Pv} = \Cv \cup D(v)$, the disjoint union $\mathop{\mathrm{Sing}}(v) \sqcup \mathrm{P}(v)$ consists of finitely many orbits, 
and the extended orbit space $S/v_{\mathrm{ex}}$ is a finite directed topological graph. 
\end{lemma}

\begin{proof}
Suppose that $v$ is a non-wandering flow without locally dense orbits.
By Lemma~\ref{lem:katok_gen}, assertion (2) holds. 
Conversely, suppose that assertion {\rm(2)} holds. 
Then $S - (D(v) \cup \mathop{\mathrm{Sing}}(v)) = \mathop{\mathrm{Per}}(v)$. 
Since the union $D(v) \cup \mathop{\mathrm{Sing}}(v)$ consists of finitely many orbits, we have 
that $S = \overline{\Pv}$, and so that $v$ is a non-wandering flow without locally dense orbits.

We claim that the extended orbit space $S/v_{\mathrm{ex}}$ is a finite directed topological graph. 
Indeed, since $\Pv$ is an open dense subset which is a surface with possible boundary, any periodic orbit has its invariant open \nbd which is either a periodic annulus or a periodic M{\"o}bius band. 
Therefore, for any connected component $U$ of $S - D(v)$, the restriction $U/v_{\mathrm{ex}} = U/v$ is an interval. 
Since any multi-saddle connection is a singleton in $S/v_{\mathrm{ex}}$, the extended orbit space $S/v_{\mathrm{ex}}$ is a finite directed topological graph. 
\end{proof}
%
%

We show Lemma~\ref{div_free_2_Ham} as follows. 

\begin{proof}[Proof of Lemma~\ref{div_free_2_Ham}]
We may assume that $S$ is connected and that $v$ is generated by a smooth divergence-free vector field $X$. 
Lemma~\ref{nw_finite_type} implies that the extended orbit space $S/v_{\mathrm{ex}}$ is a directed graph. 
Let $\pi \colon S \to S/v_{\mathrm{ex}}$ be the quotient map. 
Since any divergence-free vector field has a multi-valued Hamiltonian, the divergence-free vector field $X$ has a multi-valued Hamiltonian $\widetilde{H}$. 

Suppose that $S/v_{\mathrm{ex}}$ has no directed cycle as a directed graph. 
Considering a smooth diffeomorphism $f_e \colon \R \to \R$ with $d f_e(r)/dr = 1$ on $\R - I_e$ for some interval $I_e$, replacing $\widetilde{H}\vert_{\pi^{-1}(e)} \colon \pi^{-1}(e) \subset S \to \R$ with $f_e \circ \widetilde{H}\vert_{\pi^{-1}(e)} \colon \pi^{-1}(e) \subset S \to \R$, by the non-existence of directed cycle, we may assume that the multi-valued Hamiltonian $\widetilde{H}$ can be chosen as a single-valued Hamiltonian, which is desired. 
\end{proof}

\section{Equivalence between non-wandering property and divergence-free property}

In this section, we demonstrate the key lemma Lemma~\ref{lem:nw2div_free}. 

\subsection{Equivalence under the non-existence of non-closed recurrent orbits}
In this subsection, we show the following statement, which is a weaker statement of Lemma~\ref{lem:nw2div_free}. 

\begin{lemma}\label{nw_2_div_free}
Every non-wandering flow with finitely many singular points without locally dense orbits on a compact surface is a divergence-free flow. 
\end{lemma}

To prove the previous lemma, we present the following statements. 
First, we observe the following statement. 

\begin{lemma}\label{nw_2_div_free_no}
Every non-wandering flow with finitely many singular points without locally dense orbits nor multi-saddles on a compact surface is a divergence-free flow. 
\end{lemma}

\begin{proof}
Let $v$ be a non-wandering flow with finitely many singular points without locally dense orbits nor multi-saddles on a compact surface $S$ and $D(v)$ the multi-saddle connection diagram. 
Then $D(v) = \emptyset$.
We may assume that $S$ is connected. 
Lemma~\ref{nw_finite_type} implies that $S = \Cv$. 
By Lemma~\ref{nw_finite_type}, each connected component of the complement $S - D(v) = S$ is 
either a rotating sphere, a closed center disk, a closed periodic annulus, a periodic torus, a periodic Klein bottle, a rotating projective plane, or a closed periodic M{\"o}bius band. 
This means that the flow $v$ is topologically equivalent to an isometric flow and so is divergence-free. 
%
%
%
%
\end{proof}

We can realize \nbds of multi-saddles as the restriction of smooth Hamiltonian vector fields. 

\begin{lemma}\label{lem:multi-saddle_realization}
The restriction of a flow on a surface to a small neighborhood, called a canonical neighborhood, of a $k$-saddle is topologically equivalent to the flow generated by a smooth Hamiltonian vector field $Y$ on $\R^2$ satisfying the following properties: 
\\
{\rm(1)} If $k=0$, then all separatrices of $Y$ are $\R_{<0} \times \{0\}$ and $\R_{>0} \times \{0\}$. 
\\
{\rm(2)} If $k>0$, then every separatrix of $Y$ is of form $\{(r \cos(\pi l/2k), r \sin(\pi l/2k)) \mid r > 0\}$ for some $l \in \{ 0,1, \ldots , 2k-1 \}$. 
\end{lemma}

\begin{proof}
We claim that assertion (1) holds.
Indeed, consider a smooth bump function $\varphi \colon \R \to [0,1]$ with $\varphi^{-1}(1) = \{ 0 \}$ and $\varphi^{-1}(0) = \R - (-1,1)$ such that $\varphi$ is an even function and the restrictions $\varphi\vert_{(0,1)}$ is strictly decreasing.  
Using bump functions, we can construct a smooth odd function $g \colon \R \to [-1,1]$ with $g \vert_{\R - (-1,1)} = 0$, $\mathrm{Im}(g_y) \subseteq [-1,2]$, and $g_y^{-1}(-1) = \{0\}$ such that the restriction $g\vert_{[-1,0]}$ is a unimodal map with $g \vert_{(-1,0)} > 0$, where $g_y := dg/dy$. 
Then $g(0)= 0$. 
Define a Hamiltonian $H \colon \R^2 \to \R$ by $H(x,y) := y + \varphi(x)g(y)$. 
Then $H(x,y) = y $ for any $(x,y) \in \R^2 - (-1,1)^2$. 
Let $v_H$ be the flow generated by the Hamiltonian vector field $X_H$ of $H$. 
By $H(x,y) = y$ for any $(x,y) \in \R^2 - [-1,1]^2$, we have that $X_H \vert_{\R^2 - [-1,1]^2} = (1,0)$. 
Moreover, we obtain $H_x = g(y) d \varphi(x)/dx$ and $H_y = 1 + \varphi(x) g_y(y) \geq 1 + g_y(y) \geq 0$. 
Therefore, $H_y^{-1}(0) = \{ 0 \} \subset \R^2$. 
This means that $\mathop{\mathrm{Sing}}(v_H) = \{ 0 \}$. 
Since $H_y \vert_{\R^2 - \{ 0 \}} > 0$, the origin is the unique singular point of $v_H$, which is a $0$-saddle. 
Considering $H^{-1}(\R_{\geq 0 })$ which is homeomorphic to $\R_{\geq 0} \times \R$, the origin becomes a $\partial$-$0$-saddle. 

Fix any positive integer $k$. 
Let $\psi \colon \R_{\geq 0} \to \R$ be a strictly increasing smooth function with $\psi^{(m)}(0) = 0$ for any $m \geq 0$ and $\psi(r) = r- (1/2)$ for any $r \geq 1$. 
Moreover, define smooth Hamiltonians $f_k \colon \R^2 \to \R$ by $f_k(r,\theta) := \psi(r) \sin(2k \theta)$, where $(r,\theta)$ is the polar coordinate. 
Let $v_{f_k}$ be the flow generated by the Hamiltonian vector field $Y := X_{f_k}$ of $f_k$. 
Then $\mathop{\mathrm{Sing}}(v_{f_k}) = \{ 0 \}$ consists of one $k$-saddle. 
Since $\R \times \R_{\geq 0}$ is the union of level sets, the restriction $v_{f_k} \vert_{\R \times \R_{\geq 0}}$ has the unique $\partial$-$k/2$-saddle. 

Thus, each flow restricted to a small simply connected open \nbd of a multi-saddle is topologically equivalent to a flow generated by a smooth Hamiltonian vector field on $\R^2$. 
\end{proof}

Now, we demonstrate Lemma~\ref{nw_2_div_free} as follows. 

\begin{proof}[Proof of Lemma~\ref{nw_2_div_free}]
Let $v$ be a non-wandering flow with finitely many singular points without locally dense orbits on a compact surface $S$ and $D(v)$ the multi-saddle connection diagram. 
We may assume that $S$ is connected. 
By Lemma~\ref{nw_2_div_free_no}, we may assume that $D(v) \neq \emptyset$.
Moreover, taking the double of the surface $S$ if necessary, we may assume that $S$ is closed.
By the smoothing theorem \cite{gutierrez1986smoothing}, we also may assume that $v$ is smooth. 
Let $X$ be the smooth vector field generating $v$

Since directed surface graphs are complete invariants because of Corollary~\ref{nw_completeness}, it suffices to show that we can construct a divergence-free vector field on a closed surface by which the multi-saddle connection diagram of the flow generated is isomorphic to $D(v)$ as a directed surface graph. 
 
Lemma~\ref{nw_finite_type} implies that $S = \Cv \cup D(v)$. 
%
From the connectivity of $S$ and the non-existence of $D(v)$, Lemma~\ref{nw_finite_type} implies that each connected component of the complement $S - D(v)$ is either an open center disk, an open periodic annulus, or a periodic M{\"o}bius band. 
Lemma~\ref{lem:katok_gen} implies that each connected component of the complement $S - (D(v) \sqcup \mathop{\mathrm{Sing}}_{\mathrm{c}}(v) \sqcup \mathop{\mathrm{Per}}_{1}(v))$ is an open periodic annulus.
Let $\mathbb{A}_1, \mathbb{A}_2, \ldots , \mathbb{A}_l$ be the connected components of the complement $S - (D(v) \sqcup \mathop{\mathrm{Sing}}_{\mathrm{c}}(v) \sqcup \mathop{\mathrm{Per}}_{1}(v))$. 
For any $\lambda \in \{1,2, \ldots , l \}$, choose a periodic orbit $O_\lambda$ in the open periodic annulus $\mathbb{A}_\lambda$. 
Notice that any connected component of $S - \bigsqcup_{\lambda=1}^l O_\lambda$ which is either a center disk or a periodic M{\"o}bius band can be realized as rotating center disks and rotating periodic M{\"o}bius bands such that the restrictions near the boundaries are isometric to the flat annulus $[0,1] \times \R/\Z$ induced by the Euclidean metric. 

Fix a connected component $U$ of $S - \bigsqcup_{\lambda=1}^l O_\lambda$ intersecting a multi-saddle connection $C$. 

\begin{claim}\label{claim:09}
We may assume that there are invariant closed \nbds $V'$ and $V$ of $C$ which are closed surfaces with $V \subset \mathop{\mathrm{int}}V'$ and there is a divergence-free vector field $Y$ on $V'$ with respect to the volume form of some Riemannian metric $g_{V'}$ on $V'$ such that $W := V' - \mathop{\mathrm{int}} V$ is a finite disjoint union of invariant closed periodic annuli $A_1, \ldots , A_l$, and that there is a Hamiltonian $H_{W} \colon W \to \R$ of $Y \vert_W$ such that $W$ is a \nbd of the $\partial V'$ in $V'$. 
\end{claim}

\begin{figure}
\begin{center}
\includegraphics[scale=0.65]{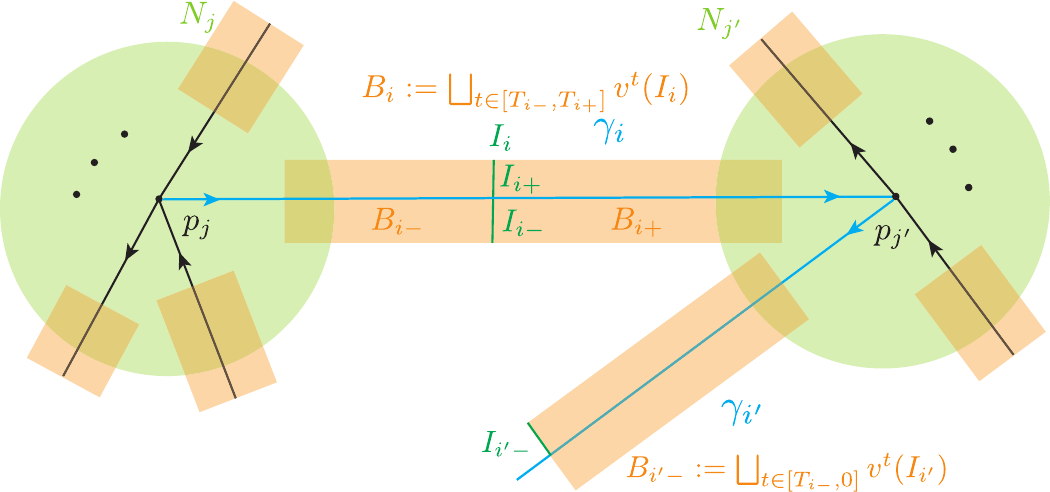}
\end{center}
\caption{A separatrix $\gamma_i$ in $C$ together with its surrounding structure.}
\label{fig:nbd_separatrix}
\end{figure}
\begin{figure}
\begin{center}
\includegraphics[scale=0.5]{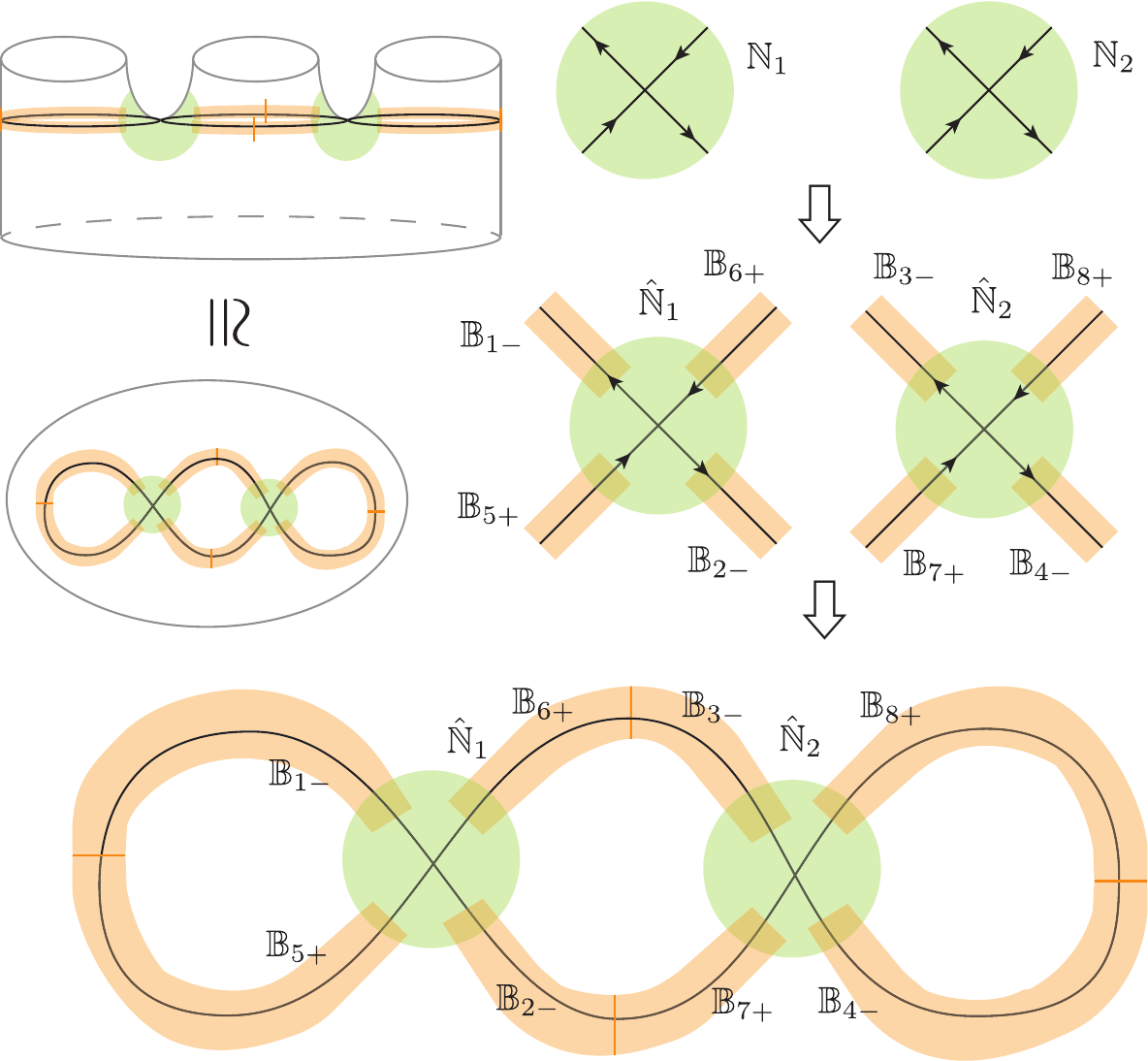}
\end{center}
\caption{An example of a \nbd $\mathbb{N}$ of a multi-saddle connection corresponding to $C$ such that $W \subset \mathbb{N}$.}
\label{fig:nbd_construction}
\end{figure}

\begin{proof}[Proof of Claim~\ref{claim:09}]
The connected component $U$ is an invariant \nbd of $C$ with $U \cap D(v) = C$. 
Let $p_1, \ldots , p_q$ be the multi-saddles in $C$ and $\gamma_1, \ldots , \gamma_k$ the separatrices in $C$. 
By Lemma~\ref{lem:multi-saddle_realization}, there are pairwise disjoint closed disks $N_1, \ldots , N_q$ which are \nbds of $p_1, \ldots , p_q$, pairwise disjoint unit closed disks $\mathbb{N}_1, \ldots , \mathbb{N}_q \subset \R^2$ which are canonical \nbds in the sense of Lemma~\ref{lem:multi-saddle_realization}, homeomorphisms $f_{N_ j} \colon N_j \to \mathbb{N}_j$, and Hamiltonians $H_{N_j} \colon \mathbb{N}_j \to \R$ respectively such that the image of $f_{N_j}$ of any orbit arc in $N_j$ is contained in the level set of $H_{N_j}$. 
Put $N_{\mathrm{m}} := \bigsqcup_{j=1}^q N_j$ and $\mathbb{N}_{\mathrm{m}} := \bigsqcup_{j=1}^q \mathbb{N}_j$. 

For any separatrix $\gamma_i$ in $C$, choose a small transverse open interval $I_i$ intersecting $\gamma_i \setminus N_{\mathrm{m}} \subset S - N_{\mathrm{m}}$ such that there are a positive number $T_{i+}$ and negative number $T_{i-}$ with $v^{T_{i-}}(I_i) \subset N_{\mathrm{m}}$ and $v^{T_{i+}}(I_i) \subset N_{\mathrm{m}}$ such that $v^t(I_i) \cap v^s(I_i) = \emptyset$ for any $t\neq s \in [T_{i-},T_{i+}]$. 
Put $B_i := \bigsqcup_{t  \in [T_{i-},T_{i+}]} v^t(I_i)$, $B_{i-} := \bigsqcup_{t  \in [T_{i-},0]} v^t(I_i)$, and $B_{i+} := \bigsqcup_{t  \in [0,T_{i+}]} v^t(I_i)$ as in Figure~\ref{fig:nbd_separatrix} . 

We will define a ``Hamiltonian'' on any $B_i$. 
Indeed, the disjoint union $B_i$ is a closed trivial flow box, and we may assume that each of $B_{i+}$ and $B_{i-}$ intersects exactly one disk which is a connected component of $N_{\mathrm{m}}$, because of the smallness of $I_i$.  
Let $I_{i-}, I_{i+}$ be the connnected components of $I_i \setminus \gamma_i$, which are transverse arcs. 
Replacing $I_i$ with a small subinterval if necessary, we may assume that either $v(I_{i\sigma}) = v(I_{i'\sigma'})$ or $v(I_{i\sigma}) \cap v(I_{i'\sigma'}) = \emptyset$ for any $i,i' \in \{ 1, \ldots , k \}$ and any $\sigma, \sigma' \in \{ -, + \}$.   
Then the union $U' := C \cup \bigsqcup_{i \in \{ 1, \ldots , k \}} v(I_{i}) = C \sqcup \bigsqcup_{i \in \{ 1, \ldots , k \}} v(I_{i-} \sqcup I_{i+})$ is a closed invariant \nbd of $C$. 
Replacing $I_i$ with a small subinterval if necessary, since the union $N_{\mathrm{m}} \cup \bigcup_{i=1}^k B_i$ is a \nbd of $C$, because the first return maps to $I_{i\pm}$ is identical or involutive, we may assume that $U' \subset N_{\mathrm{m}} \cup \bigcup_{i=1}^k B_i$. 
Fix any small number $\varepsilon_{B}>0$. 
Define a diffeomorphism $f_{B_i} \colon B_i \to \mathbb{B}_{i} := [T_{i-},T_{i+}] \times [-\varepsilon_{B},\varepsilon_{B}]$ with $f_{B_i}(B_i \cap \gamma_i) = [T_{i-},T_{i+}] \times \{ 0 \}$ and $f_{B_i}(v^t(I_i)) = \{t\} \times [-\varepsilon_{B},\varepsilon_{B}]$ such that the image of any orbit arc via $f_{B_i}$ is contained in $[T_{i-},T_{i+}] \times \{ y \}$ for some $y \in [-\varepsilon_{B},\varepsilon_{B}]$. 
Write $\mathbb{B}_{i-} := [T_{i-},0] \times [-\varepsilon_{B},\varepsilon_{B}] = f_{B_i}(B_{i-})$, $\mathbb{B}_{i+} := [0,T_{i+}] \times [-\varepsilon_{B},\varepsilon_{B}] = f_{B_i}(B_{i+})$, $\mathbb{B}'_{i-} := [T_{i-}/\delta,0] \times [-\varepsilon_{B},\varepsilon_{B}] \subset \mathbb{B}_{i-}\setminus \mathbb{N}_m$, and $\mathbb{B}'_{i+} := [0,T_{i+}/\delta] \times [-\varepsilon_{B},\varepsilon_{B}] \subset \mathbb{B}_{i+} \setminus \mathbb{N}_m$ for some $\delta > 1$. 
Moreover, define a smooth function $H_{\pm B_{i}} \colon \mathbb{B}_{i\pm} \to \R$ by $H_{\pm B_{i}}(x,y) := \pm y$. 

We will ``extend'' the Hamiltonians $H_{N_j}$ to $H_{\mathbb{N}_j}$ using $B_{i_n}$. 
Indeed, for any integer $j \in \{1, \ldots , q \}$, let $\gamma_{i_1}, \ldots , \gamma_{i_{j_r}}$ be the separatrices from $p_j$ and $\gamma_{i_{j_r +1}}, \ldots , \gamma_{i_{2 j_r}}$ be the separatrices to $p_j$, and put $\hat{\mathbb{N}}_j := \mathbb{N}_j \sqcup \bigsqcup_{n=1}^{j_r} \mathbb{B}_{i_n -} \sqcup \bigsqcup_{n=j_r + 1}^{2 j_r} \mathbb{B}_{i_n +}$ and $\widetilde{\mathbb{N}}_j := \hat{\mathbb{N}}_j/\mathop{\sim}$, where the equivalent relation $\sim$ on $\hat{\mathbb{N}}_j$ is defined as follows: 
$a \sim b$ if one of the following statements holds: 
\\
(1) $a=b \in \hat{\mathbb{N}}_j$. 
\\
(2) There are an integer $n \in \{1, \ldots , 2 j_r\}$ and a sign $\sigma \in \{ - , + \} $ such that $a \in \mathbb{N}_j$, $b \in \mathbb{B}_{i_n \sigma}$, and $f_{N_j}^{-1}(a) = f_{B_{i_n}}^{-1}(b) \in N_j \cap B_{i_n \sigma}$. 
\\
(3) There are an integer $n \in \{1, \ldots , 2 j_r\}$ and a sign $\sigma \in \{ - , + \} $ such that $b \in \mathbb{N}_j$, $a \in \mathbb{B}_{i_n \sigma}$, and $f_{N_j}^{-1}(b) = f_{B_{i_n}}^{-1}(a) \in N_j \cap B_{i_n \sigma}$. 
\\
Put $\mathbb{I}_{i_n -} := f_{N_j}(v^{T_{i_n -}}(I_{i_n}))$ and $\mathbb{I}_{i_n +} := f_{N_j}(v^{T_{i_n +}}(I_{i_n}))$. 
Identify subsets $\mathbb{N}_j$, $\mathbb{B}'_{i_n -}$, $\mathbb{B}'_{i_n +}$, $\mathbb{B}_{i_n -}$, $\mathbb{B}_{i_n +}$, $\mathbb{I}_{i_n -}$, and $\mathbb{I}_{i_n +}$ with subsets of $\widetilde{\mathbb{N}}_j$. 
For any $\mathbb{B}_{i_n -}$ (resp. $\mathbb{B}_{i_n +}$), choose $\mu_{i_n} \in \{ -, + \}$ such that the smooth function $H_{\mu_{i_n} B_{i_n}} \colon \mathbb{B}_{i_n -} \to \R$ satisfies that either both $H_{\mu_{i_n} B_{i_n}} \vert_{\mathbb{I}_{i_n -}}$ (resp. $H_{\mu_{i_n} B_{i_n}} \vert_{\mathbb{I}_{i_n +}}$) and $H_{N_j} \vert_{\mathbb{I}_{i_n -}}$ are increasing, or both of them are decreasing. 
Since $\bigsqcup_{n=1}^{j_r} \mathbb{B}'_{i_n -} \sqcup \bigsqcup_{n=j_r + 1}^{2 j_r} \mathbb{B}'_{i_n +} \subset \widetilde{\mathbb{N}}_j - \mathbb{N}_j \subset \bigsqcup_{n=1}^{j_r} \mathbb{B}_{i_n -} \sqcup \bigsqcup_{n=j_r + 1}^{2 j_r} \mathbb{B}_{i_n +}$, there is a smooth function $\varphi_j \colon \widetilde{\mathbb{N}}_j \to [0,1]$ with $\bigsqcup_{n=1}^{j_r} \mathbb{B}'_{i_n -} \sqcup \bigsqcup_{n=j_r + 1}^{2 j_r} \mathbb{B}'_{i_n +} \subseteq \varphi^{-1}(1)$ and $\mathbb{N}_j \subseteq \varphi^{-1}(0)$ such that the restrictions $\varphi_j \vert_{\mathbb{B}_{i_n -}} \colon [T_{i_n -},0] \times [-\varepsilon_{B},\varepsilon_{B}] \to [0,1]$ and $\varphi_j \vert_{\mathbb{B}_{i_n +}} \colon [0,T_{i_n +}] \times [-\varepsilon_{B},\varepsilon_{B}] \to [0,1]$ depend only on the first coordinate and is weakly monotonic with respect to the first coordinate. 
Using the smooth function $\varphi_j$, define a Riemannian metric $g_{\widetilde{\mathbb{N}}_j}$ on $\widetilde{\mathbb{N}}_j$ satisfying that thre restriction $g_{\widetilde{\mathbb{N}}_j} \vert_{\mathbb{N}_j}$ (resp. $g_{\widetilde{\mathbb{N}}_j} \vert_{\mathbb{B}'_j}$) is the Euclidean metric on $\mathbb{N}_j \subset \R^2$ (resp. $\mathbb{B}'_j \subset \R^2$). 
Write $\mathbb{N} := \bigcup_{j=1}^q \widetilde{\mathbb{N}}_j$ (see an example as shown in Figure~\ref{fig:nbd_construction}). 
Then the intersection $\widetilde{\mathbb{N}}_{j'} \cap \widetilde{\mathbb{N}}_j$ for any $j' \neq j$ consists of pairwise disjoint transverse closed intervals in $\bigsqcup_{i=1}^k I_i$. 
Therefore the Riemannian metrics $g_{\widetilde{\mathbb{N}}_j}$ on $\widetilde{\mathbb{N}}_j$ induce a Riemannian metric $g_{\mathbb{N}}$ on $\mathbb{N}$. 
The function $H_{\widetilde{\mathbb{N}}_j} \colon \widetilde{\mathbb{N}}_j \to \R$ defined by $H_{\widetilde{\mathbb{N}}_j} \vert_{\mathbb{N}_j} = H_{N_j} \vert_{\mathbb{N}_j}$ and 
\[
H_{\widetilde{\mathbb{N}}_j} \vert_{\mathbb{B}'_{i_n\sigma}}(x,y) = \varphi(x) H_{\mu_{i_n} B_{i_n}}(x,y) + (1 - \varphi(x)) H_{N_j}(x,y)
\]
is well-defined and smooth, where $\sigma \in \{ -, + \}$. 
Then $H_{\widetilde{\mathbb{N}}_j}\vert_{\mathbb{B}'_{i_n \sigma}} = H_{\mu_{i_n} B_{i_n}}\vert_{\mathbb{B}'_{i_n \sigma}}$. 

We will find an invariant closed \nbd $V' \subset \mathbb{N}$ which is the union of ``level sets'' of multi-valued Hamiltonian $\{ H_{\widetilde{\mathbb{N}}_j} \}_i$ in this claim. 
Indeed, since the restrictions $H_{\mu_{i_n} B_{i_n}} \vert_{\mathbb{B}_{i_n\sigma}}(x,y)$ depend only on the second coordinate $y$ and either both of them are increasing or both of them are decreasing, so is the restriction $H_{\widetilde{\mathbb{N}}_j}\vert_{\mathbb{B}_{i_n\sigma}}(x,y)$. 
Since $\mathbb{N}_j$ is canonical, by taking $I_{i_n}$ small if necessary, we may assume that $\partial (H_{N_j}\vert_{\mathbb{N}_j \cap \mathbb{B}_{i_n\sigma}})/\partial y$ is either positive or negative for the local coordinate system $(x,y)$ on $\mathbb{B}_{i_n\sigma}$. 
Therefore the smooth function $H_{\widetilde{\mathbb{N}}_j}$ has a unique critical point $f_{N_j}(p_j) \in \mathbb{N}_j \subseteq \widetilde{\mathbb{N}}_j$. 
Let $X_{H_{\widetilde{\mathbb{N}}_j}}$ be the Hamiltonian vector field with respect to the volume form of the Riemannian metric $g_{\widetilde{\mathbb{N}}_j}$.  
By the construction of $H_{\widetilde{\mathbb{N}}_j}$, the vector field $X_{H}$ on $\mathbb{M}$ defined by $X_H \vert_{\widetilde{\mathbb{N}}_j} = X_{H_{\widetilde{\mathbb{N}}_j}}$ is well-defined and a smooth divergence-free vector field. 
Since $\varepsilon_{B}$ is small, the integrable curve of any point in $I_{i_n} \setminus \gamma_{i_n}$ is a periodic orbit. 
Therefore the union $V' := \bigcup_{i=1}^q H_{\widetilde{\mathbb{N}}_j}^{-1}([-\varepsilon_{B},\varepsilon_{B}]) \subseteq \bigcup_{i=1}^q \widetilde{\mathbb{N}}_j = \mathbb{N}$ is a closed \nbd of $C$ such that any connected component of the complement $V' - C$ is an annulus. 
Then the flow $v_{V'}$ on the compact surface $V'$ generated by $X_{H_{\widetilde{\mathbb{N}}_j}}$ satisfies that the multi-saddle connection diagram of $v_{V'}$ coincides with $C$ as directed surface graphs on the compact surfaces $V'$ and $U'$, and that the complement $V' - C$ is a finite disjoint union of invariant periodic annuli. 
Write $g_{V'} := g_{\mathbb{N}}\vert_{V'}$. 

Let $V := \bigcup_{i=1}^q H_{\widetilde{\mathbb{N}}_j}^{-1}([-\varepsilon_{B}/2,\varepsilon_{B}/2]) \subset V'$ be a closed \nbd of $C$ such that $V - C$ is a finite disjoint union of invariant periodic annuli. 
The restriction $Y := X_{H} \vert_V$ satisfies that $v_Y := v_{V'} \vert_V$ is generated by $Y$. 
Then the set diffrence $W := V' - \mathop{\mathrm{int}} V$ is a finite disjoint union of invariant closed periodic annuli. 
By the construction of $H_{\widetilde{\mathbb{N}}_j}$, the function $H_W \colon W \to \R$ defined by $H_W\vert_{\widetilde{\mathbb{N}}_j \cap W} := \vert H_{\widetilde{\mathbb{N}}_j} \vert$ on any $j \in \{ 1, \ldots , q \}$ is well-defined and smooth such that the restriction $v_Y \vert_{W}$ is the flow generated by the Hamiltonian vector field of $H_W$ with respect to the volume form of the Riemannian metrics $g_{\mathbb{N}}$.  
 \end{proof}

For any $\lambda \in \{ 1, \ldots , l \}$, since $\partial A_\lambda \subseteq \partial V \cup \partial V'$, denote by $\partial_0 A_\lambda$ the boundary component of $A_\lambda$ contained in $\partial V$. 
By the uniqueness of closed annulus up to diffeomorphism (cf. \cite[Theorem~9.3.11]{hirsch2012differential}), we can choose diffeomorphisms $h_\lambda \colon A_\lambda \to \mathbb{A} := [0,1] \times \R/\Z$ with $h_\lambda(\partial_0 A_\lambda) = \{0 \} \times \R/\Z$ such that the first component of the restriction $h_{\lambda*}Y\vert_{\{0 \} \times \R/\Z}$ of the pushout is positive. 
Write $A_{\lambda+} := h_\lambda^{-1}([2/3,1] \times \R/\Z)$. 
Then $h_\lambda(A_\lambda - A_{\lambda+})$ is a \nbd of $\{0 \} \times \R/\Z = h_\lambda(\partial_0 A_\lambda) \subset \partial \mathbb{A}$.
Reparametrizing the first coordinate $[0,1]$ of $\mathbb{A}$ by using a diffeomorphism $F \colon [0,1] \to [0,1]$ such that the derivative of $F$ near boundary $\{0,1\}$ is one if necessary, we may assume that the first component $(h_{\lambda*}Y)_x$ of the restriction of the pushout $h_{\lambda*}Y = ((h_{\lambda*}Y)_x, (h_{\lambda*}Y)_y)$ on the \nbd $h_\lambda(A_\lambda - A_{\lambda+})$ of $\{0 \} \times \R/\Z = h_\lambda(\partial_0 A_\lambda)$ is positive.
This means that the composition $H_W \circ h_\lambda^{-1} \vert_{h_\lambda(A_\lambda - A_{\lambda+})} \colon h_\lambda(A_\lambda - A_{\lambda+}) = [0,2/3) \times \R/\Z  \to \R$ satisfies that $\partial (H_W \circ h_\lambda^{-1})/\partial y$ on $[0,2/3) \times \R/\Z$ is positive. 
Let $\psi \colon [0,1] \to [0,1]$ be a smooth function with $\psi^{-1}(1) = [0,1/3]$ and $\psi^{-1}(0) = [2/3,1]$ such that the restriction $\psi\vert_{[1/3,2/3]}$ is strictly decreasing. 

\begin{claim}\label{claim:06}
We may assume that there is a small \nbd of $\partial V'$ each of whose connected components is isometric to $\mathbb{A} = [0,1] \times \R/\Z$. 
\end{claim}

\begin{proof}[Proof of Claim~\ref{claim:06}]
Denote by $g_{V'}$ (resp. $g_{\mathbb{A}}$) the Riemannian metric on $V'$ in the previous claim (resp. the flat metric  on $\mathbb{A} = [0,1] \times \R/\Z$ induced by $g_{\R^2}$). 
Define a Riemannian metric $g$ on $V'$ by $g\vert_{V} = g_{V'}\vert_{V}$ and 
\[
g\vert_{A_{\lambda}}(h^{-1}(x,y)) = \psi(x)g_{V'}(h^{-1}(x,y)) + (1-\psi(x))h^*(g_{\mathbb{A}}(x,y))
\]
where $h^*g_{\mathbb{A}}$ is the pullback of $g_{\mathbb{A}}$. 
Then $g = h^*g_{\mathbb{A}}$ near $\partial V'$ in $V'$. 
\end{proof}

\begin{claim}\label{claim:10}
We may assume that there is a small \nbd of $\partial V'$ in $W$ of which the restriction to any connected component is the restriction of the isometric rotation of $\mathbb{A}$. 
\end{claim}

\begin{proof}[Proof of Claim~\ref{claim:10}]
Define a smooth function $h_W \colon W \to \R$ as follows: 
\[
h_W \vert_{A_\lambda} (h_\lambda^{-1}(x,y)) := \psi(x)H_W(h_\lambda^{-1}(x,y)) + (1-\psi(x))y
\]
Then $h_W\vert_{A_\lambda} (h_\lambda^{-1}(x,y)) = y$ on $A_{\lambda+}$. 
Since $\partial (H_W \circ h_\lambda^{-1})/\partial y$ on $[0,2/3) \times \R/\Z$ is positive, so is the partial derivative $\partial (h_W \circ h_\lambda^{-1})/\partial y$ on $[0,2/3) \times \R/\Z$. 
This means that the partial derivative $\partial (h_W \circ h_\lambda^{-1})/\partial y$ on $h_\lambda(A_\lambda)$ is positive and so that the Hamiltonian vector field $Y'$ of $h_W$ with the Riemaniann volume of the Riemannian metric $g_{V'}$ is non-singular on $A_\lambda$ for any $\lambda \in \{ 1, \ldots , l \}$. 
Replacing $Y$ with $Y'$, the assertion holds. 
\end{proof}

By construction, the restriction $v\vert_{S - \bigsqcup_{\lambda=1}^k U_{O_\lambda}}$ is topologically equivalent to the flow generated by the Hamiltonian vector field $Y$ for any small open periodic annulus $U_{O_\lambda}$ which is a \nbd of $O_\lambda$ respectively. 
Identifying the pairs of boundary components of $W$ such that the multi-saddle connection diagram of the resulting flow $v'_Y$ from $v_Y$ becomes isomorphic to $D(v)$ as a directed surface graph, the resulting flow $v'_Y$ is divergence-free and is topologically equivalent to $v$. 
This means that the flow $v$ is divergence-free. 
\end{proof}

\subsection{Equivalence in general case}

Roughly speaking, Lemma~\ref{lem:nw2div_free} is proved as follows:
By Lemma~\ref{lem:katok_gen}(2), the surface $S$ can be obtain by  gluing two invariant open subsets $\mathop{\mathrm{Sing}}_{\mathrm{c}}(v) \sqcup \Pv$, $\mathrm{LD}(v)$ and any small \nbd of $D(v)$. 
Therefore, by modifying the restrictions on the three components so that the flows are divergence-free and gluing the resulting multi-valued Hamiltonians together via a partition of unity, we show that $v$ is topologically equivalent to a divergence-free flow. 
More precisely, we establish Lemma~\ref{lem:nw2div_free} as follows.


\begin{proof}[Proof of Lemma~\ref{lem:nw2div_free}]
Fix any Riemannian metric $g$ on $S$. 
By Lemma~\ref{lem:katok_gen}(2), any connected component of $S -(\mathop{\mathrm{Sing}}_{\mathrm{c}}(v) \sqcup D(v))$ is an invariant open subset contained in either $\Pv$ or $\mathrm{LD}(v)$. 
Fix any small \nbd $U$ of $D(v)$. 
By the same argument of the proof of Claim~\ref{claim:09}, we may assume that the restriction of $v$ to the union $W := U \cup \Pv \subset S - \mathrm{LD}(v)$ is divergence-free.  
Then $S - W \subset \mathrm{LD}(v)$. 
By using a homeomorphism on $S$, we may assume that there is a multi-valued Hamiltonian $H_W \colon W \to \R$ whose divergence-free vector field $X_{H_W}$ equals the derivative $\partial v/\partial t$.   
Taking $W$ small if necessary, we may assume that $U \cap \overline{\mathrm{LD}(v)}$ is a collar of $\partial \mathrm{LD}(v)$ in $\overline{\mathrm{LD}(v)}$ which is homeomorphic to $\partial \mathrm{LD}(v) \times [0,1)$. 
We identify $U$ with $\partial \mathrm{LD}(v) \times [0,1)$. 
Then $\partial \mathrm{LD}(v) \times \{ 0 \} = \partial \mathrm{LD}(v)$. 
Fix any connected component $V$ of $S -(\mathop{\mathrm{Sing}}_{\mathrm{c}}(v) \sqcup D(v))$ intersectiong $\mathrm{LD}(v)$.
Since any locally dense orbit intersects a closed transversal (cf. \cite[Lemma~2.1.3]{nikolaev1999flows}, \cite[Corollary 3.14]{yokoyama2021poincare}) by the waterfall construction (cf. \cite[Lemma~3.3.7 p.86]{Candel2000foliation}), there is a closed transversal $C \subset V$ with $v(C) = V$. 
By the structure theorem of Gutierrez \cite{gutierrez1986smoothing} (see also \cite[Lemma 3.9]{gutierrez1986smoothing} and \cite[Theorem 2.5.1]{nikolaev1999flows}), for the closed transversal $C$ for the Q-set $\overline{V}$, the first return map on $C$ is topologically conjugate to a minimal interval exchange transformation possibly with flips on a circle. 
%
%
Using the intervals in $C$, by the flow box theorem (cf. \cite[Theorem~4.2.6, p.95]{markley2023flows}, \cite[Theorem 1.1, p.45]{aranson1996introduction}), there is a multi-valued Hamiltonian $h \colon S- (\partial \mathrm{LD}(v) \times [0,1/2)) \to \R$ such that $g(X_{H_W},X_h) > 0$ on $S- (\partial \mathrm{LD}(v) \times [0,1/2))$, where $X_h$ is the divergence-free vector field on $S- (\partial \mathrm{LD}(v) \times [0,1/2))$. 
Take a smooth bump function $\varphi \colon [0,1] \to [0,1]$ with $\varphi^{-1}(1) = [5/6,1]$ and $\varphi^{-1}(0) = [0,2/3]$ such that $\varphi$ is an even function and the restrictions $\varphi\vert_{(2/3,5/6)}$ is strictly decreasing.  
Define a multi-valued Hamiltonian $H \colon S \to \R$ by $H = H_W$ on $W - (\partial \mathrm{LD}(v) \times (1/2,1))$, $H = h$ on $S -W$, and 
\[
H (x,y) = \varphi(y) h(x,y) + (1 - \varphi(y)) H_{W}(x,y)
\]
on $\partial \mathrm{LD}(v) \times [0,1/2)$. 
By construction, the divergence-free vector field of $H$ generates a flow which is topologically equivalent to $v$. 
This means that $v$ is divergence-free.
\end{proof}

\section{Examples}

In the last section, we provided an example to illustrate the necessity of the non-existence of directed cycles and the case of a non-wandering flow that may be non-Hamiltonian.

\subsection{Examples of non-Hamiltonian area-preserving flows}\label{example}

Irrational rotations on a torus are non-Hamiltonian area-preserving flows. 
Any flows on a torus consisting of periodic orbits are non-Hamiltonian area-preserving flows. 

The following example is a non-Hamiltonian area-preserving flow without closed transversals, which M. Asaoka suggested to show the necessity of the non-existence of directed cycles in Theorem~\ref{nw_ham}. 
Indeed, there is a smooth non-Hamiltonian area-preserving flow with non-degenerate singular points without either locally dense orbits or closed transversals on a torus as in Figure~\ref{non_dc}.

\vspace{10pt}

\begin{figure}
\begin{center}
\includegraphics[scale=0.4]{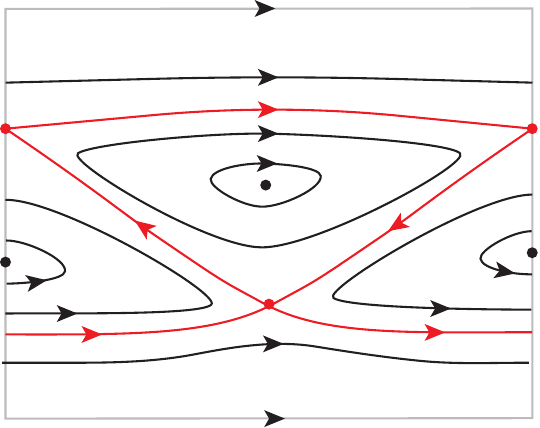}
\end{center}
\caption{A smooth non-Hamiltonian area-preserving flow with four non-degenerate singular points without closed transversals on a torus that consists of two centers, two saddles, and periodic orbits.}
\label{non_dc}
\end{figure}

{\bf Acknowledgement}: 
The author would like to thank Masashi Asaoka for his motivating question and helpful comments. 

\bibliographystyle{amsplain}
\bibliography{../yt20211011}

\providecommand{\bysame}{\leavevmode\hbox to3em{\hrulefill}\thinspace}
\providecommand{\MR}{\relax\ifhmode\unskip\space\fi MR }
\providecommand{\MRhref}[2]{%
  \href{http://www.ams.org/mathscinet-getitem?mr=#1}{#2}
}
\providecommand{\href}[2]{#2}
\begin{thebibliography}{10}

\bibitem{aranson1996introduction}
S~Kh Aranson, Genrikh~Ruvimovich Beliski{\u\i}, and EV~Zhuzhoma,
  \emph{Introduction to the qualitative theory of dynamical systems on
  surfaces}, American Mathematical Society, 1996.

\bibitem{arnol1991topological}
Vladimir~I Arnol'd, \emph{Topological and ergodic properties of closed 1-forms
  with incommensurable periods}, Functional Analysis and Its Applications
  \textbf{25} (1991), no.~2, 81--90.

\bibitem{asaoka2016}
M.~Asaoka, personal communication, June 24 personal communication, 2016.

\bibitem{bolsinov1999exact}
Aleksei~Viktorovich Bolsinov and Anatoly~Timofeevich Fomenko, \emph{Exact
  topological classification of hamiltonian flows on smooth two-dimensional
  surfaces}, Journal of Mathematical Sciences \textbf{94} (1999), no.~4,
  1457--1476.

\bibitem{Candel2000foliation}
Alberto Candel and Lawrence Conlon, \emph{Foliations. {I}}, Graduate Studies in
  Mathematics, vol.~23, American Mathematical Society, Providence, RI, 2000.
  \MR{1732868}

\bibitem{chaika2021singularity}
Jon Chaika, Krzysztof Fr{\k{a}}czek, Adam Kanigowski, and Corinna Ulcigrai,
  \emph{Singularity of the spectrum for smooth area-preserving flows in genus
  two and translation surfaces well approximated by cylinders}, Communications
  in Mathematical Physics \textbf{381} (2021), no.~3, 1369--1407.

\bibitem{cobo2010flows}
Milton Cobo, Carlos Gutierrez, and Jaume Llibre, \emph{Flows without wandering
  points on compact connected surfaces}, Transactions of the American
  Mathematical Society \textbf{362} (2010), no.~9, 4569--4580.

\bibitem{conze2011cocycles}
Jean-Pierre Conze and Krzysztof Fr{\k{a}}czek, \emph{Cocycles over interval
  exchange transformations and multivalued hamiltonian flows}, Advances in
  Mathematics \textbf{226} (2011), no.~5, 4373--4428.

\bibitem{forni1997solutions}
Giovanni Forni, \emph{Solutions of the cohomological equation for
  area-preserving flows on compact surfaces of higher genus}, Annals of
  Mathematics \textbf{146} (1997), no.~2, 295--344.

\bibitem{forni2002deviation}
\bysame, \emph{Deviation of ergodic averages for area-preserving flows on
  surfaces of higher genus}, Annals of Mathematics \textbf{155} (2002), no.~1,
  1--103.

\bibitem{frkaczek2012ergodic}
Krzysztof Fr{\k{a}}czek and Corinna Ulcigrai, \emph{Ergodic properties of
  infinite extensions of area-preserving flows}, Mathematische Annalen
  \textbf{354} (2012), no.~4, 1289--1367.

\bibitem{gutierrez1986smoothing}
Carlos Guti{\'e}rrez, \emph{Smoothing continuous flows on two-manifolds and
  recurrences}, Ergodic Theory and dynamical systems \textbf{6} (1986), no.~1,
  17--44.

\bibitem{hirsch2012differential}
Morris~W Hirsch, \emph{Differential topology}, vol.~33, Springer Science \&
  Business Media, 2012.

\bibitem{kanigowski2016ratner}
Adam Kanigowski and Joanna Ku{\l}aga-Przymus, \emph{Ratner’s property and
  mild mixing for smooth flows on surfaces}, Ergodic Theory and Dynamical
  Systems \textbf{36} (2016), no.~8, 2512--2537.

\bibitem{katok1973invariant}
Anatolii~Borisovich Katok, \emph{Invariant measures of flows on oriented
  surfaces}, Doklady Akademii Nauk, vol. 211, Russian Academy of Sciences,
  1973, pp.~775--778.

\bibitem{kulaga2012self}
Joanna Ku{\l}aga, \emph{On the self-similarity problem for smooth flows on
  orientable surfaces}, Ergodic Theory and Dynamical Systems \textbf{32}
  (2012), no.~5, 1615--1660.

\bibitem{ma2005geometric}
Tian Ma and Shouhong Wang, \emph{Geometric theory of incompressible flows with
  applications to fluid dynamics}, no. 119, American Mathematical Soc., 2005.

\bibitem{markley2023flows}
Nelson~G Markley and Mary Vanderschoot, \emph{Flows on compact surfaces: The
  weil--hedlund--anosov program}, Springer Nature, 2023.

\bibitem{marzougui2002area}
Habib Marzougui, \emph{Area preserving flows with a dense orbit}, Nonlinearity
  \textbf{15} (2002), no.~5, 1379.

\bibitem{Nikolaenko20}
Stanislav~Sergeevich Nikolaenko, \emph{Topological classification of
  hamiltonian systems on two-dimensional noncompact manifolds}, Sbornik
  Mathematics \textbf{211} (2020), no.~8, 1127--1158.

\bibitem{nikolaev1998finite}
Igor Nikolaev, \emph{Finite groups and flows on 2-manifolds}, Nonlinearity
  \textbf{11} (1998), no.~2, 213.

\bibitem{nikolaev2001non}
\bysame, \emph{Non-wandering flows on the 2-manifolds}, Journal of Differential
  Equations \textbf{173} (2001), no.~1, 1--16.

\bibitem{nikolaev2013foliations}
\bysame, \emph{Foliations on surfaces}, vol.~41, Springer Science \& Business
  Media, 2013.

\bibitem{nikolaev1999flows}
Igor Nikolaev and Evgeny Zhuzhoma, \emph{Flows on 2-dimensional manifolds: an
  overview}, no. 1705, Springer Science \& Business Media, 1999.

\bibitem{novikov1982hamiltonian}
Sergei~Petrovich Novikov, \emph{{The Hamiltonian formalism and a many-valued
  analogue of Morse theory}}, Uspekhi Matematicheskikh Nauk \textbf{37} (1982),
  no.~5, 3--49.

\bibitem{Oshemkov10}
Andrey~Alexandrovich Oshemkov, \emph{{Classification of hyperbolic
  singularities of rank zero of integrable Hamiltonian systems}}, Sbornik
  Mathematics \textbf{201} (2010), no.~8, 1153--1191.

\bibitem{poincare1881memoire}
Henri Poincar{\'e}, \emph{M{\'e}moire sur les courbes d{\'e}finies par une
  {\'e}quation diff{\'e}rentielle (i)}, Journal de math{\'e}matiques pures et
  appliqu{\'e}es \textbf{7} (1881), 375--422.

\bibitem{ravotti2017quantitative}
Davide Ravotti, \emph{Quantitative mixing for locally hamiltonian flows with
  saddle loops on compact surfaces}, {Annales Henri Poincar{\'e}}, vol.~18,
  Springer, 2017, pp.~3815--3861.

\bibitem{sakajo2018tree}
T~Sakajo and T~Yokoyama, \emph{Tree representation of topological streamline
  patterns of structurally stable 2d hamiltonian vector fields in multiply
  conected domains}, The IMA Journal of Applied Mathematics \textbf{83} (2018),
  380--411.

\bibitem{sakajo2015transitions}
Takashi Sakajo and Tomoo Yokoyama, \emph{Transitions between streamline
  topologies of structurally stable hamiltonian flows in multiply connected
  domains}, Physica D: Nonlinear Phenomena \textbf{307} (2015), 22--41.

\bibitem{ulcigrai2011absence}
Corinna Ulcigrai, \emph{Absence of mixing in area-preserving flows on
  surfaces}, Annals of mathematics (2011), 1743--1778.

\bibitem{yokoyama2016topological}
Tomoo Yokoyama, \emph{A topological characterization for non-wandering surface
  flows}, Proceedings of the American Mathematical Society \textbf{144} (2016),
  no.~1, 315--323.

\bibitem{yokoyama2017decompositions}
\bysame, \emph{Decompositions of surface flows}, arXiv preprint
  arXiv:1703.05501 (2017).

\bibitem{yokoyama2017genericity}
\bysame, \emph{Genericity for non-wandering surface flows}, Journal of
  Dynamical and Control Systems \textbf{23} (2017), no.~2, 197--212.

\bibitem{yokoyama2021generalization}
\bysame, \emph{Generalization of poincar$\backslash$'e recurrence theorem for
  flows on surfaces and characterization of minimal flows on compact surfaces},
  arXiv preprint arXiv:2110.05705 (2021).

\bibitem{yokoyama2023topological}
\bysame, \emph{Topological characterizations of recurrence, poisson stability,
  and isometric property of flows on surfaces}, Israel Journal of Mathematics
  (2023), 1--25.

\bibitem{yokoyama2021poincare}
\bysame, \emph{{A Poincar{\'e}-Bendixson theorem for flows with arbitrarily
  many singular points}}, Studia Mathematica \textbf{278} (2024), 99--172.

\bibitem{yokoyama2013word}
Tomoo Yokoyama and Takashi Sakajo, \emph{Word representation of streamline
  topologies for structurally stable vortex flows in multiply connected
  domains}, Proceedings of the Royal Society A: Mathematical, Physical and
  Engineering Sciences \textbf{469} (2013), no.~2150, 20120558.

\bibitem{zorich1999leaves}
Anton Zorich, \emph{How do the leaves of a closed {$1$}-form wind around a
  surface?}, Pseudoperiodic topology, Amer. Math. Soc. Transl. Ser. 2, vol.
  197, Amer. Math. Soc., Providence, RI, 1999, pp.~135--178. \MR{1733872}

\end{thebibliography}

\end{document}